\pdfoutput=1
\documentclass[a4paper]{amsart}
\usepackage{ifdraft}
\PassOptionsToPackage{final,hidelinks}{hyperref}
\usepackage{nima}
\usepackage{amsaddr}
\usepackage{tikz-cd}
\ifdraft{\usepackage{screen}}{}
\usepackage{xnotes}

\DeclareMathOperator{\CAT}{CAT}

\DeclareMathOperator{\GL}{GL}
\DeclareMathOperator{\Isom}{Isom}
\DeclareMathOperator{\IsomAut}{IsomAut}

\DeclareMathOperator{\tor}{tor}
\DeclareMathOperator{\AsCone}{AsCone}

\theoremstyle{definition}
\newtheorem*{cons*}{Construction}

\newcommand{\E}{\mathbb{E}}
\renewcommand{\sth}{:}

\newcommand{\eqnlabel}[1]{\label{eqn:#1}}
\newcommand{\eqnref}[1]{\ref{eqn:#1}}
\newcommand{\peqnref}[1]{(\eqnref{#1})}

\title{Crystallographic Helly Groups}
\author{Nima Hoda}
\address{DMA, École normale supérieure \\ Université
  PSL, CNRS \\ 75005 Paris, France
  \\ \ \\ Instytut Matematyczny,
  Uniwersytet Wroc\l awski\\
  pl.\ Grun\-wal\-dzki 2/4,
  50--384 Wroc{\l}aw, Poland}
\email{nima.hoda@mail.mcgill.ca}
\date{\today}


\begin{document}

\begin{abstract}
  We prove that asymptotic cones of Helly graphs are countably
  hyperconvex.  We use this to show that virtually nilpotent Helly
  groups are virtually abelian and to characterize virtually abelian
  Helly groups via their point groups.  In fact, we do this for the
  more general class of coarsely injective spaces and groups.  We
  apply this to prove that the $3$-$3$-$3$-Coxeter group is not Helly
  (nor even coarsely injective), thus obtaining the first example of a
  systolic group that is not Helly, answering a question of Chalopin,
  Chepoi, Genevois, Hirai, and Osajda.
\end{abstract}

\maketitle

\tableofcontents

\section{Introduction}

A graph $\Gamma$ is \defterm{Helly} if any collection of pairwise
intersecting metric balls in the vertex set $\Gamma^0$ of $\Gamma$ has
nonempty total intersection.  A group $G$ is \defterm{Helly} if it
acts properly and cocompactly on a Helly graph \cite{Chalopin:2020}.
Such an action has very strong consequences for $G$ including that $G$
is biautomatic, that $G$ has finitely many conjugacy classes of finite
subgroups, that there is a finite model for $\underline{E}G$ and that
$G$ admits an EZ-boundary \cite{Chalopin:2020}.  The class of Helly
groups is vast: it includes Gromov hyperbolic groups \cite{Lang:2013,
  chepoi:packing:2017, Chalopin:2020}, cocompactly cubulated groups
\cite{Bandelt:1991, Polat:2001, Roller:1998, Gerasimov:1998,
  Chepoi:2000}, finitely presented graphical $C(4)$-$T(4)$-small
cancellation groups \cite{Chalopin:2020}, quadric groups
\cite{Chalopin:2020}, Artin groups of FC-type and weak Garside groups
of finite type \cite{Huang:2019}.

Helly graphs are the integer metric analog of hyperconvex metric
spaces, which are equivalent to injective metric spaces
\cite{Aronszajn:extension}.  While the ball intersection property that
defines hyperconvex metric spaces is more useful for us here, we will
use the terminology \emph{injective metric space} since it is more
common in the geometric group theory literature.  A metric space $X$
is \defterm{injective} if and only if $X$ is geodesic and every
pairwise intersecting collection of closed balls in $X$ has nonempty
total intersection.  Every metric space isometrically embeds in a
canonical smallest injective metric space \cite{Isbell:1964}.
Injective metric spaces are complete, they have strong fixed point
properties \cite{espinola_khamsi:hyperconvex} and they admit
equivariant geodesic bicombings \cite{Lang:2013}.  Every Helly group
acts properly and cocompactly on an injective metric space
\cite{Chalopin:2020}.

The main contribution of this paper is the following theorem.  See
\Ssecref{point_group} for the definition of the point group.

\begin{mainthm}[\Thmref{main}]
  \mainthmlabel{main} Let $G$ be a finitely generated virtually
  abelian group with point group $P' < \GL_n(\R)$.  If $G$ is Helly
  (or just coarsely injective) then $P'$ is conjugate to a subgroup
  of $\IsomAut\bigl(\R^n,|\cdot|_{\infty}\bigr)$, the group of
  isometric automorphisms of $\R^n$ with the supremum norm
  $|\cdot|_{\infty}$.  Consequently, $G$ acts properly and cocompactly
  by isometries on $\bigl(\R^n,|\cdot|_{\infty}\bigr)$.
\end{mainthm}

Hagen proved the consequent of \Thmref{main} for $G$ cocompactly
cubulated \cite{Hagen:2014}.  Thus \Thmref{main} is a strengthening of
Hagen's result to the more general class of Helly groups.  Moreover,
as Hagen also proves a converse in the cocompactly cubulated case, we
see that \Thmref{main} fully characterizes the virtually abelian Helly
groups and we have the following.

\begin{mainthm}[\Thmref{main}]
  \mainthmlabel{equivalences} Let $G$ be a virtually abelian group.  The
  following conditions are equivalent.
  \begin{enumerate}
  \item $G$ is Helly.
  \item $G$ is coarsely injective.
  \item $G$ is cocompactly cubulated.
  \end{enumerate}
\end{mainthm}

In more recent work, Petyt and Spriano show that for the class of
crystallographic groups, being hierarchically hyperbolic is also
equivalent to being cocompactly cubulated
\cite{petyt_spriano:unbounded:2023}.

As an application of \Mainthmref{main}, we prove that the
$3$-$3$-$3$-Coxeter group is not Helly, thus obtaining a first example
of a systolic group that is not Helly and answering a question of
Chalopin, Chepoi, Genevois, Hirai, and Osajda \cite{Chalopin:2020}.

We also prove the following result, which is already a consequence of
semihyperbolicity of coarsely injective groups
\cite{Alonso_Bridson:1995:semihyperbolic, Lang:2013}.

\begin{mainthm}[\Thmref{nilpotent}]
  \mainthmlabel{nilpotent} Let $G$ be a Helly (or coarsely injective)
  group.  If $G$ is virtually nilpotent then $G$ is virtually abelian.
\end{mainthm}

Our proofs of \Mainthmref{main} and \Mainthmref{nilpotent} apply
results of Pansu on asymptotic cones of nilpotent groups
\cite{Pansu:1983} and results of Isbell on injective metric spaces
\cite{Isbell:1964}.  Along the way we prove the following theorem
which may be of independent interest.

\begin{mainthm}[\Thmref{chc}]
  \mainthmlabel{chc} Let $Y$ be a coarsely injective space (e.g. a
  Helly graph) and let $X$ be an asymptotic cone of $Y$.  Then $X$ is
  countably hyperconvex.
\end{mainthm}

A metric space $X$ is \defterm{countably hyperconvex} if it is
geodesic and every pairwise intersecting \emph{countable} collection
of closed balls in $X$ has nonempty total intersection.  Countable
hyperconvexity is also known as $\aleph_1$-hyperconvexity
\cite{Aronszajn:extension}.

\subsection{Structure of the paper}

In \Secref{preliminaries} we give a few basic definitions from rough
metric geometry.  In \Secref{virtually_zn} we show some fundamental
properties of finitely generated virtually abelian groups and review
classical results relating them to crystallographic groups.  In
\Secref{zn_metrics} we characterize the proper, roughly geodesic,
left-invariant metrics on $\Z^n$.  In
\Secref{ascones_and_coarse_injective} we define and review basic
properties of asymptotic cones and prove \Mainthmref{chc}.  In
\Secref{nilpotent_injective} we prove \Mainthmref{nilpotent}.  In
\Secref{abelian_helly} we prove \Mainthmref{main} and
\Mainthmref{equivalences}.

\subsection{Acknowledgements}

The author would like to thank Victor Chepoi for his excellent course
on Helly groups at the 2019 Simons Semester in Geometric and Analytic
Group Theory in Warsaw.  The author is also thankful to Piotr
Przytycki and Damian Osajda for their encouragement in pursuing this
research, to Urs Lang for his helpful suggestions on the virtually
nilpotent case and to Pierre Pansu for a very informative discussion
on asymptotic cones of nilpotent groups.  This work was partially
supported by the ERC grant GroIsRan, the Polish Narodowe Centrum Nauki
UMO-2017/25/B/ST1/01335 as well as by the grant 346300 for IMPAN from
the Simons Foundation and the matching 2015-2019 Polish MNiSW fund.

\section{Preliminaries}
\seclabel{preliminaries}

In this section we give a few basic definitions from rough metric
geometry.

Let $X$ and $Y$ be metric spaces, let $S$ be a set and let $C$ be a
positive real.  A function $f \colon S \to Y$ is $C$-roughly onto if
every $y \in Y$ is at distance at most $C$ from some point in
$f(X) \subset Y$.  Two functions $f_1,f_2 \colon S \to Y$ are
\defterm{$C$-roughly equivalent} if
\[ d\bigl(f_1(s),f_2(s)\bigr) \le C \] for every $s \in S$.  A
\defterm{$C$-rough isometric embedding} from $X$ to $Y$ is a function
$f \colon X \to Y$ such that
\[ d(x_1,x_2) - C \le d\bigl(f(x_1),f(x_2)\bigr) \le d(x_1,x_2) + C \]
for all $x_1,x_2 \in X$.  A $C$-rough isometric embedding
$f \colon X \to Y$ is a \defterm{$C$-rough isometry} if it is roughly
onto.  Note that two metrics $d_1$ and $d_2$ on $X$ are roughly
equivalent (as functions $X \times X \to \R$) if and only if the
identity map on $S$ is a rough isometry from $(X,d_1)$ to $(X,d_2)$.

A \defterm{$C$-rough geodesic} in $X$ from $x_1$ to $x_2$ is a
$C$-rough isometric embedding $f$ from the interval
$\bigl[0,\ell\bigr] \subset \R$ to $X$ with $\ell = d(x_1,x_2)$ such
that $f(0) = x_1$ and $f(\ell) = x_2$.  A metric space $(X,d)$ is
\defterm{$C$-roughly geodesic} if every pair of points in $X$ is
joined by a $C$-rough geodesic.

\section{Virtually \texorpdfstring{$\Z^n$}{Z\textasciicircum n}
  groups}
\seclabel{virtually_zn}

In this section we show some fundamental properties of finitely
generated virtually abelian groups.  We define the point group and
review classical results relating crystallographic groups and finitely
generated virtually abelian groups.

Let $G$ be a virtually $\Z^n$ group and let $A < G$ be a finite index
subgroup isomorphic to $\Z^n$.  By replacing $A$ with its normal core
$\bigcap _{g \in G} gAg^{-1}$ we may assume that $A$ is a normal
subgroup.  Hence, choosing some identification of $A$ with $\Z^n$, we
have a short exact sequence of groups
\[ 1 \to \Z^n \to G \to P \to 1 \] with $P$ finite.

\subsection{The point group}
\sseclabel{point_group}

Let $G \to \GL_n(\Z)$ be the conjugation action $g \cdot z = gzg^{-1}$
of $G$ on $\Z^n$.  The image $P'$ of this map, viewed as a subgroup of
$\GL_n(\R)$, is the \defterm{point group} of $G$.  Note that $P'$ is
finite since the action $G \to \GL_n(\Z)$ factors through $P$.  The
point group of $G$ is well defined up to conjugation by elements of
$\GL_n(\R)$.  That is, if we choose a different finite index normal
subgroup $A' < G$ and some identification of $A'$ with $\Z^{n'}$ then
$n' = n$ and the image of $G$ in $\GL_n(\Z)$ is conjugate to $P'$ by
an element of $\GL_n(\R)$.  To show this, since $A' \cap A$ has finite
index in both $A'$ and $A$, it suffices to consider the case where
$A' < A$.  In this case we have two $G$ actions on $\Z^n$ and a
$G$-equivariant inclusion $\Z^n \hookrightarrow \Z^n$.  But this
extends to a $G$-equivariant isomorphism $f \colon \R^n \to \R^n$ and
we see that the $G$-actions differ by conjugation by
$f \in \GL_n(\R)$.

\subsection{Crystallographic groups}

\seclabel{crystals} The $G$-action on $\Z^n$ extends to a $G$-action
on $\R^n$ such that the inclusion $\Z^n \hookrightarrow \R^n$ is
$G$-equivariant.  So, by \Propref{sesmap}, we have a unique group
homomorphism $G \to G'$ and short exact sequence
\[ 1 \to \R^n \to G_{\R} \to P \to 1 \] such that the diagram
\[ \begin{tikzcd}
    1 \ar[r] & \Z^n \ar[r] \ar[d,hook] & G \ar[r] \ar[d] & P \ar[r] \ar[d, "1_P"] & 1 \\
    1 \ar[r] & \R^n \ar[r] & G_{\R} \ar[r] & P \ar[r] & 1
  \end{tikzcd} \] commutes.  Since $P$ is finite and $\R^n$ is
divisible and torsion-free, we have $H^k(P , \R^n) = 0$ for all $k$.
So, by the classification of group extensions via second cohomology
\cite[Chapter~IV]{brown:cohomology}, the projection $G_{\R} \to P$ has
a section and so $G_{\R} \isomor \R^n \rtimes P$.  The image $P'$ of
the action $P \to \GL_n(\R)$ is the point group of $G$.  Since
$P' < \GL_n(\R)$ is finite, for some $\psi \in \GL_n(\R)$, we have
$\psi P' \psi^{-1} < O(n)$.  So there is an isomorphism
$G_{\R} \to \R^n \rtimes_{\alpha} P$ where $\alpha \colon P \to O(n)$
is defined by $\alpha_p(r) = \psi\bigl(p \cdot \psi^{-1}(r)\bigr)$.
But the isometry group $\Isom(\E^n)$ of $n$-dimensional Euclidean
space $\E^n$ splits as $\R^n \rtimes O(n)$, where the $\R^n$ factor
acts by translations and the $O(n)$ factor acts linearly.  So there is
a map $\R^n \rtimes_{\alpha} P \to \Isom(\E^n)$ such that the diagram
\[ \begin{tikzcd}
    1 \ar[r] & \R^n \ar[r] \ar[d, "\psi"] & G_{\R} \ar[r] \ar[d, "\isomor"] & P \ar[r] \ar[d, "1_P"] & 1 \\
    1 \ar[r] & \R^n \ar[r] \ar[d, "1_{\R^n}"] & \R^n \rtimes_{\alpha} P \ar[r] \ar[d] & P \ar[r] \ar[d] & 1 \\
    1 \ar[r] & \R^n \ar[r] & \Isom(\E^n) \ar[r] & O(n) \ar[r] & 1
  \end{tikzcd} \] commutes.  Composing the maps
$G \to G_{\R} \to \R^n \rtimes_{\alpha} P \to \E^n$ we obtain a map
$G \to \Isom(\E^n)$ that is injective on $\Z^n$ and so has finite
kernel $K$.  The resulting action of $G$ on $\E^n$ is cocompact since
$G \to \Isom(\E^n)$ sends $\Z^n$ to a lattice in the translation group
$\R^n$.  Since, additionally, the group $P$ is finite, the action of
$G$ on $\E^n$ is proper.  Groups acting properly, cocompactly and
faithfully on $\E^n$ are known as \defterm{crystallographic groups}.
We have shown that any finitely generated virtually abelian group has
a quotient by a finite normal subgroup that is crystallographic.  This
was first proved by Zassenhaus \cite{zassenhaus:crystallographic}.

Conversely, given a crystallographic group $G < \Isom(\E^n)$, the map
$G \to O(n)$ induced by the projection $\R^n \rtimes O(n) \to O(n)$
has finite image and has kernel isomorphic to $\Z^n$
\cite{bieberbach:crystallographic1, bieberbach:crystallographic2}.
Thus $G$ is virtually abelian and has the image of $G \to O(n)$ as its
point group.

\section{Metrics on \texorpdfstring{$\Z^n$}{Z\textasciicircum n} and norms}
\seclabel{zn_metrics}

In this section we use a result of Burago and Fujiwara to characterize
the proper, roughly geodesic, left-invariant metrics on $\Z^n$, up to
rough isometry.

We denote the \defterm{discrete metric} on any set by the following
notation.
\[ \delta(x_1, x_2) =
  \begin{cases}
    1 & \text{if $x_1 \neq x_2$} \\
    0 & \text{otherwise}
  \end{cases} \]

An action of a group $G$ on a metric space $Y$ is \defterm{metrically
  proper} if for any $y \in Y$ and $r \ge 0$ the set
$\{g \in G \sth d(gy,y) \le r \}$ is finite.  The action of $G$ on $Y$
is \defterm{cobounded} if there exists $r \ge 0$ such that for any
$y,y' \in Y$ there exists $g \in G$ satisfying $d(gy,y') \le r$.

\begin{prop}
  \proplabel{pbmetric} Let $G$ be a group acting metrically properly
  and coboundedly on a geodesic metric space $Y$.  Then, for any
  choice of basepoint $y_0 \in Y$, the expression
  \[ d_G(g_1,g_2) = d_{Y}(g_1\cdot y_0,g_2\cdot y_0) +
    \delta(g_1, g_2) \] defines a proper, roughly geodesic,
  left-invariant metric on $G$.  Choosing a different basepoint
  results in a roughly equivalent metric.

  Moreover the map
  \begin{align*}
    (G,d) &\to Y \\
    g &\mapsto g\cdot y_0
  \end{align*}
  is a rough isometry.
\end{prop}
\begin{proof}
  The expression $d_{Y}(g_1\cdot y_0,g_2\cdot y_0)$ defines a
  pseudometric $d'_G$ on $G$.  Then $d_G = d'_G + \delta$ is a metric
  since the sum of a pseudometric and a metric is a metric.  Since $G$
  acts by isometries on $Y$, the pseudometric $d'_G$ is left-invariant
  on $G$ and thus the metric $d_G$ is also left-invariant.  Since the
  $G$-action on $Y$ is metrically proper, balls in $(G, d_G)$ are
  finite and so $d_G$ is a proper metric.

  Since
  \[ \bigl| d_G(g_1, g_2) - d_{Y}(g_1\cdot y_0,g_2\cdot y_0) \bigr| =
    \delta(g_1, g_2) \le 1 \] the map $g \mapsto g \cdot y_0$ is a
  rough isometric embedding.  The $G$-action on $Y$ is cobounded, so
  the map $g \mapsto g \cdot y_0$ is also roughly onto and is thus a
  rough isometry.  Then, since rough geodesicity is a rough isometry
  invariant, the metric $d_G$ is roughly geodesic.

  Finally, if $y_1 \in Y$ is a different basepoint and $d^1_G$
  is the resulting metric on $G$, we have
  \begin{align*}
    d^1_G(g_1, g_2)
    &= d_{Y}(g_1\cdot y_1,g_2\cdot y_1) + \delta(g_1, g_2) \\
    &\le d_{Y}(g_1\cdot y_1, g_1\cdot y_0)
      + d_{Y}(g_1\cdot y_0,g_2\cdot y_0)
      + d_{Y}(g_2\cdot y_0, g_2\cdot y_1) + \delta(g_1, g_2) \\
    &= d_{Y}(y_1, y_0)
      + d_{Y}(g_1\cdot y_0,g_2\cdot y_0) + \delta(g_1, g_2)
      + d_{Y}(y_0, y_1) \\
    &= d_G(g_1, g_2)
      + 2d_{Y}(y_0, y_1)
  \end{align*}
  and by a symmetric argument we also have
  $d_G(g_1, g_2) \le d^1_G(g_1, g_2) + 2d_{Y}(y_0, y_1)$ so $d_G$
  and $d^1_G$ are roughly equivalent.
\end{proof}

\begin{lem}
  \lemlabel{zmetric} Let $d$ be a metric on $\Z$ such that the
  identity map is a $(K,A)$-quasi-isometry from $\Z$ with the standard
  metric to $(\Z, d)$.  Assume that $d$ is roughly invariant in the
  sense that there exists a uniform $C$ such that
  \[ d(m,n) - C \le d(k+m,k+n) \le d(m,n) + C \] for any
  $k,m,n \in \Z$.  Then $\frac{d(0,n)}{|n|}$ tends to a limit
  $L \in \bigl[\frac{1}{K},K\bigr]$ as $n$ tends to $\pm \infty$.
\end{lem}
\begin{proof}
  By rough invariance
  \[ \frac{d(0,|n|)}{|n|} - \frac{C}{|n|} \le \frac{d(0,n)}{|n|} \le
    \frac{d(0,|n|)}{|n|} + \frac{C}{|n|} \] so it suffices to consider
  the case where $n \to +\infty$.

  We have
  \[ \frac{n}{K} - A \le d(0,n) \le Kn + A \] so that
  \[ \frac{1}{K} - \frac{A}{n} \le \frac{d(0,n)}{n} \le K +
    \frac{A}{n} \] and so
  \[ \frac{1}{K} \le \liminf_n \frac{d(0,n)}{n} \le \limsup_n
    \frac{d(0,n)}{n} \le K \] and it remains only to prove that
  $\liminf_n \frac{d(0,n)}{n} \ge \limsup_n \frac{d(0,n)}{n}$.

  Let $\epsilon > 0$.  Take $m$ large enough that
  $m < \frac{1}{\epsilon}$ and
  $\frac{d(0,m)}{m} < \liminf_n \frac{d(0,n)}{n} + \epsilon$.  Let
  $M = \max_{0 \le n \le m}d(0,m)$.  Then for any positive $n$ we have
  \begin{align*}
    d(0,n) &\le d(0,m) + d(m,2m) + \cdots + d\bigl((q-1)m,qm\bigr) + M \\
           &\le q\bigl(d(0,m) + C\bigr) + M \\
           &\le \frac{n}{m}\bigl(d(0,m) + C\bigr) + M
  \end{align*}
  where $q = \bigl\lfloor \frac{n}{m} \bigr\rfloor$.  So
  \[ \frac{d(0,n)}{n} \le \frac{d(0,m)}{m} + \frac{C}{m} + \frac{M}{n}
    < \liminf_n \frac{d(0,n)}{n} + \epsilon + C\epsilon +
    \frac{M}{n} \] which implies
  $\limsup_n \frac{d(0,n)}{n} \le \liminf_n \frac{d(0,n)}{n} +
  (1+C)\epsilon$.  But $\epsilon$ was chosen arbitrarily so we have
  $\limsup_n \frac{d(0,n)}{n} \le \liminf_n \frac{d(0,n)}{n}$ as
  needed.
\end{proof}

The following theorem was proved by Fujiwara \cite{Fujiwara:2015} by
modifying a proof of Burago \cite{Burago:1992}.

\begin{thm}[{Burago, Fujiwara \cite[Corollary~3.3]{Fujiwara:2015}}]
  \thmlabel{fujiwara} Let $d_1$ and $d_2$ be two proper, roughly
  goedesic, left-invariant metrics on $\Z^n$.  If
  $\lim_{m\to\infty}\frac{d_1(0,mz)}{d_2(0,mz)} = 1$ for every
  $z \in \Z^n$ then $d_1$ and $d_2$ are roughly equivalent.
\end{thm}

\begin{prop}
  \proplabel{mtonorm} Let $d$ be a proper, roughly geodesic,
  left-invariant metric on $\Z^n$.  Then there exists a norm
  $|\cdot|_d$ on $\R^n$ such that the inclusion
  $(\Z^n,d) \hookrightarrow (\R^n, |\cdot|_d)$ is a rough
  isometry.

  Moreover, if $f \in \GL_n(\Z)$ is a rough isometry of $(\Z^n,d)$
  then $f$, viewed as an element of $\GL_n(\R)$, is an isometric
  automorphism of $(\R^n,|\cdot|_d)$.
\end{prop}

\begin{rmk*}
  \Propref{mtonorm} implies that any asymptotic cone of $(\Z^n,d)$ is
  isometric to $(\R^n,|\cdot|_d)$, regardless of the choice of
  ultrafilter, basepoint sequence or scaling sequence.
\end{rmk*}

\begin{proof}[Proof of \Propref{mtonorm}]
  By the Milnor-Schwarz Lemma, the inclusion
  $(\Z^n,d) \hookrightarrow (\R^n,|\cdot|_2)$ is a
  $(K,A)$-quasi-isometry for some $K$ and $A$.  Let $s \in \R^n$ with
  $|s|_2 = 1$.  For each $m \in \Z$ choose $z_m \in \Z^n$ with
  $|z_m - ms|_2 \le C$ for some uniform $C$.  Call $(z_m)_m$ an
  \defterm{integral sequence} for $s$.  The expression
  \[ d_{\Z}(m,k) = d(z_m,z_k) + \delta(m, k) \] defines a metric
  $d_{\Z}$ on $\Z$.  We claim that (I) the identity map on $\Z$ is a
  quasi-isometry from $\Z$ to $(\Z,d_{\Z})$ with multiplicative
  constant $K$ and that (II) $d_{\Z}$ is roughly invariant so that we
  may apply \Lemref{zmetric}.

  To prove (I), we extend $ms \mapsto z_m$ to a quasi-inverse $\phi$
  of the inclusion $(\Z^n,d) \hookrightarrow (\R^n, |\cdot|_2)$ by
  arbitrarily choosing $\phi(x)$ in
  $\Z^n \cap B_{\frac{\sqrt{n}}{2}}(x)$.  Then $\phi$ is a
  quasi-isometry with the same multiplicative constant $K$ as the
  inclusion $(\Z^n,d) \hookrightarrow (\R^n, |\cdot|_2)$.  The map
  $\Z \to \{z_m \sth m \in \Z\}$ given by $m \mapsto z_m$ is the
  composition of the isometric embedding $\Z \to (\R^n, |\cdot|_2)$
  given by $m \mapsto ms$ with the quasi-isometry $\phi$.  Then
  $m \mapsto z_m$ is a quasi-isometry from $\Z$ to the subspace
  $\{z_m \sth m \in \Z\}$ of $(\Z^n,d)$.  The metric $d_{\Z}$ is, up
  to rough equivalence, pulled back along this quasi-isometry and so
  we have (I).

  To prove (II), take $m, k, \ell \in \Z$.  It suffices to show that
  \[ d(z_{m+\ell},z_{k+\ell}) \le d(z_m,z_k) + B \] for some uniform
  constant $B$.  We will first show that the distance between
  $z_{k+\ell} - z_{m+\ell}$ and $z_k - z_m$ is uniformly bounded.
  We have that $z_{k+\ell} - z_{m+\ell}$ is equal to
  \[ \bigl(z_{k+\ell} - (k+\ell)s\bigr) + \bigl((k+\ell)s -
    (m+\ell)s\bigr) + ((m+\ell)s - z_{m+\ell}) \] and $z_k - z_m$ is
  equal to \[ (z_k - ks\bigr) + (ks - ms) + (ms - z_m) \] but
  $\bigl((k+\ell)s - (m+\ell)s\bigr) = (ks - ms)$ and the remaining
  terms have norm bounded by $C$.  Hence
  \[ \bigl|(z_{k+\ell} - z_{m+\ell}) - (z_k - z_m)\bigr|_2 \le 4C \]
  and so
  \begin{align*}
    d(z_{m+\ell},z_{k+\ell}) &\le d(z_{m+\ell}, z_k + z_{m+\ell} - z_m)
    + d(z_k + z_{m+\ell} - z_m, z_{k+\ell}) \\
    &= d(z_m, z_k) + d(z_k - z_m, z_{k+\ell} - z_{m+\ell}) \\
    &\le d(z_m, z_k) + K\bigl|(z_{k+\ell} - z_{m+\ell})
      - (z_k - z_m)\bigr|_2 + KA \\
    &\le d(z_m, z_k) + 4KC + KA
  \end{align*}
  so we can take $B = 4KC + KA$.

  Applying \Lemref{zmetric} to $(\Z,d_{\Z})$ we see that
  $\frac{d(0,z_m)}{|m|}$ tends to a limit
  $L_{(z_m)_m} \in \bigl[\frac{1}{K},K\bigr]$ as $m$ tends to
  $\pm \infty$.  This limit depends only on $s$.  Indeed, given a
  different integral sequence $(z'_m)_{m\in \Z}$, letting
  $(z''_m)_{m\in \Z}$ be the integral sequence
  \[
    (\ldots,z_{-4},z'_{-3},z_{-2},z'_{-1}z_0,z'_1,z_2,z'_3,z_4,\ldots) \]
  we see that $L_{(z_m)_m} = L_{(z''_m)_m} = L_{(z'_m)_m}$.  Moreover,
  the limit depends only on $\{\pm s\}$ since $(z_{-m})_m$ is an
  integral sequence for $-s$ and $L_{(z_{-m})_m} = L_{(z_m)_m}$.  We
  will write $L_{\pm s}$ for $L_{(z_m)_m}$.

  For $r \in \R^n \setminus \{0\}$ set $|r|_d = L_{\pm s}|r|_2$, where
  $s = \frac{r}{|r|_2}$.  For $0 \in \R^n$ set $|0|_d = 0$.  We claim
  that $|\cdot|_d$ is a norm on $\R^n$.  We need to prove that, for
  any $r,r' \in \R^n$ and $\alpha \in \R$, the following properties
  hold.
  \begin{enumerate}
  \item \itmlabel{nonneg} $|r|_d \ge 0$
  \item \itmlabel{zero} $|r|_d = 0$ iff $r = 0$
  \item \itmlabel{mult} $|\alpha r|_d = |\alpha| |r|_d$
  \item \itmlabel{triangle} $|r + r'|_d \le |r|_d + |r'|_d$
  \end{enumerate}
  Properties \pitmref{nonneg}, \pitmref{zero} and \pitmref{mult}
  follow immediately from the properties of $L_{\pm s}$ and the
  definition of $|\cdot|_d$.  We will prove
  Property~\pitmref{triangle}.  Define $s$, $s'$ and $s''$ by
  \[ s = \frac{r}{|r|_2} \qquad s' = \frac{r'}{|r'|_2} \qquad s'' =
    \frac{r+r'}{|r+r'|_2}\] and choose integral sequences $(z_m)_m$,
  $(z'_m)_m$ and $(z''_m)_m$ for $s$, $s'$ and $s''$ as above.  For
  each $m \in \N$ choose $a_m$ and $a'_m$ so that the following
  inequalities hold.
  \[ \biggl|\frac{|r|_2}{|r+r'|_2} - \frac{a_m}{m}\biggr|_2 \le \frac{1}{m}
    \qquad
    \biggl|\frac{|r'|_2}{|r+r'|_2} - \frac{a'_m}{m}\biggr|_2 \le \frac{1}{m} \]
  Then
  \begin{align*}
    ms'' &= \frac{m(r+r')}{|r+r'|_2} \\
         &= \frac{m\bigl(s|r|_2+s'|r'|_2\bigr)}{|r+r'|_2} \\
         &= \frac{|r|_2}{|r+r'|_2}ms + \frac{|r'|_2}{|r+r'|_2}ms' \\
         &= a_ms + a'_ms'
           + \biggl(\frac{|r|_2}{|r+r'|_2} - \frac{a_m}{m}\biggr)ms
           + \biggl(\frac{|r'|_2}{|r+r'|_2} - \frac{a'_m}{m}\biggr)ms'
  \end{align*}
  and so $\bigl|ms'' - (a_ms + a'_ms')\bigr|_2 \le 2$.  Then, since
  \[ z''_m - (z_{a_m} + z'_{a'_m}) = \bigl(ms'' - (a_ms+ a'_ms')\bigr) + (z''_m -
    ms'') - (z_{a_m} - a_ms) - (z'_{a'_m} - a'_ms') \] we have
  \[ \bigl|z''_m - (z_{a_m} + z'_{a'_m})\bigr|_2 \le 2 + 3C \] and so
  \begin{align*}
    d(0,z''_m) &\le d(0,z_{a_m}) + d(z_{a_m},z_{a_m}+ z'_{a'_m})
    + d(z_{a_m} + z'_{a'_m}, z''_m) \\
    &= d(0,z_{a_m}) + d(0,z'_{a'_m}) + d(z_{a_m}+z'_{a'_m}, z''_m) \\
    &\le d(0,z_{a_m}) + d(0,z'_{a'_m})
      + K\bigl|z''_m - (z_{a_m}+z'_{a'_m})\bigr|_2 + KA \\
    &\le d(0,z_{a_m}) + d(0,z'_{a'_m})
      + K(2 + 3C + A)
  \end{align*}
  which will allow us to estimate $L_{\pm s''}$ using $L_{\pm s}$ and
  $L_{\pm s'}$.  Dividing both sides of this inequality by $m$ we
  have
  \[ \frac{d(0,z''_m)}{m} \le
    \frac{a_m}{m}\cdot\frac{d(0,z_{a_m})}{a_m} +
    \frac{a'_m}{m}\cdot\frac{d(0,z_{a'_m})}{a'_m} + \frac{K(2 + 3C +
      A)}{m} \] and taking the limit as $m$ tends to $\infty$ we have
  \[ L_{\pm s''} \le \frac{|r|_2}{|r+r'|_2}L_{\pm s} +
    \frac{|r'|_2}{|r+r'|_2}L_{\pm s'} \] and so
  $|r+r'|_d = L_{\pm s''}|r+r'|_2 \le L_{\pm s}|r|_2 + L_{\pm
    s'}|r'|_2 = |r|_d + |r'|_d$ thus proving
  Property~\pitmref{triangle}.

  We prove now that the inclusion
  $(\Z^n,d) \hookrightarrow (\R^n,|\cdot|_d)$ is a rough isometry.
  This will follow from \Thmref{fujiwara} once we show that
  $\lim_{k\to\infty} \frac{d(0,kz)}{|kz|_d} = 1$ for
  $z \in \Z^n \setminus \{0\}$.  Since
  $|kz|_d = k|z|_d = kL_{\pm s}|z|_2$, where $s = \frac{z}{|z|_2}$, it
  suffices to show that $\frac{d(0,kz)}{k|z|_2}$ tends to $L_{\pm s}$
  as $k$ tends to $\infty$.  We define an integral sequence
  $(z_m)_{m\in\Z}$ for $s$ by setting $z_m = kz$ when
  $m = \lfloor k|z|_2\rfloor$ and arbitrarily choosing
  $z_m \in B_{\frac{\sqrt{n}}{2}}(ms)\cap\Z^n$ when
  $m \notin \bigl\{\lfloor k|z|_2\rfloor \sth k \in \Z\bigr\}$.  Then
  $\Bigl(\frac{d(0,kz)}{\lfloor k|z|_2\rfloor}\Bigr)_k$ is a
  subsequence of $\Bigl(\frac{d(0,z_m)}{m}\Bigr)_m$ and so tends to
  the same limit $L_{\pm s}$.  But
  $\lim_k \frac{k|z|_2}{\lfloor k|z|_2\rfloor} = 1$ and so we have
  $\lim_k \frac{d(0,kz)}{k|z|_2} = L_{\pm s}$ as required.

  It remains to prove that if $f \in \GL_n(\Z)$ is a rough isometry of
  $(\Z^n,d)$ that $f$, viewed as an element of $\GL_n(\R)$, is an
  isometry of $\bigl(\R^n,|\cdot|_d\bigr)$.  Take $r \in \R^n$, let
  $s = \frac{r}{|r|_2}$, let $s' = \frac{f(r)}{|f(r)|_2}$ and let
  $(z_m)_m$ be an integral sequence for $s$.  Since $f$ is a
  bilipschitz map from $\bigl(\R^n,|\cdot|_2\bigr)$ to itself there is
  a uniform bound $C$ on $\bigl|f(z_m-ms)\bigr|_2$.  Since
  $r = |r|_2s$, we have $s' = \frac{f(s)}{|f(s)|_2}$ and so
  \[ f(z_m-ms) = f(z_m) - mf(s) = f(z_m) - m\bigl|f(s)\bigr|_2s' \]
  Let $k_m = \bigl\lfloor m|f(s)|_2\bigr\rfloor$.  Then
  \[ \bigl|m|f(s)|_2s' - k_ms'\bigr|_2 = \bigl|m|f(s)|_2 - k_m\bigr|
    \le 1 \] and so $\bigl|f(z_m) - k_ms'\bigr|_2 \le C + 1$ for all
  $m \in \Z$.  Hence there is an integral sequence $(z'_k)_k$ for $s'$
  obtained by setting $z'_k = z_m$ when $k = k_m$ and then extending
  to all $k$.  Then $(z'_{k_m})_m$ is a subsequence of $(z'_k)_k$ so
  we have
  \[ L_{\pm s'} = \lim_{m\to\infty} \frac{d(0,z'_{k_m})}{k_m} =
    \lim_{m\to\infty} \frac{d(0,z_m)}{\bigl\lfloor
      m|f(s)|_2\bigr\rfloor} = \lim_{m\to\infty}
    \frac{d(0,z_m)}{m|f(s)|_2} = \frac{L_{\pm s}}{|f(s)|_2}\] so we
  have
  $\bigl|f(r)\bigr|_d = L_{\pm s'}\bigl|f(r)\bigr|_2 = L_{\pm
    s}\frac{|f(r)|_2}{|f(s)|_2} = L_{\pm
    s}\frac{|r|_2|f(s)|_2}{|f(s)|_2} = |r|_d$.
\end{proof}

\section{Asymptotic cones and coarsely injective spaces}
\seclabel{ascones_and_coarse_injective}

In this section we define and review basic properties of asymptotic
cones.  We then prove that any asymptotic cone of a Helly graph is
countably hyperconvex.  In fact, we prove something stronger: that the
asymptotic cone of a coarsely injective space is countably
hyperconvex.

A metric space $X$ is \defterm{$C$-coarsely injective} if, for any
family $\{(x_i,r_i)\}_{i \in I}$ of pairs
$(x_i,r_i) \in X \times \R_{\ge 0}$ for which
$d(x_i,x_{i'}) \le r_i + r_{i'}$ for all $i,i' \in I$, there exists
$x \in X$ for which $d(x,x_i) \le r_i + C$ for all $i \in I$.  A
metric space is \defterm{injective} if it is $0$-coarsely injective.
A metric space is \defterm{countably hyperconvex} if it satisfies the
$0$-coarse injectivity condition for families
$\{(x_i,r_i)\}_{i \in I}$ that are countable.  A graph $\Gamma$ is
\defterm{Helly} if its vertex set $\Gamma^0$ endowed with the graph
metric satisfies the $0$-coarse injectivity condition for families
$\{(x_i,r_i)\}_{i \in I}$ whose every $r_i$ is an integer.

\begin{rmk*}
  Any injective metric space is countably hyperconvex.  Any countably
  hyperconvex metric space is complete and geodesic.  Any coarsely
  injective metric space is roughly geodesic.
\end{rmk*}
  
\begin{rmk*}
  Any Helly graph is $1$-coarsely injective.
\end{rmk*}

For an exposition on asymptotic cones of metric spaces see Drutu and
Kapovich \cite[Chapter~10]{Drutu:2018}.  Let $(X,d)$ be a metric
space.  Let $\mathscr{U}$ be a nonprincipal ultrafilter on $\N$.  Let
$(b^{(n)})_{n\in\N}$ be a sequence in $X$.  Let $(s^{(n)})_{n\in\N}$
be a sequence of positive reals such that $s^{(n)} \to \infty$ as
$n \to \infty$.  Consider the set
\[ X^{\N}_{(b^{(n)})_n,(s^{(n)})_n} = \Biggl\{(x_n)_{n\in\N} \sth
  \text{$\biggl(\frac{d(x_n,b^{(n)})}{s^{(n)}}\biggr)_{n\in\N}$ is
    bounded} \Biggr\} \] of \defterm{sequences in $X$ that are bounded
  with respect to} the \defterm{basepoint sequence} $(b^{(n)})$ and
the \defterm{scaling sequence} $(s^{(n)})_n$.  For
$(x_n)_n, (x'_n)_n \in X^{\N}_{(b^{(n)})_n,(s^{(n)})_n}$,
\[ d\bigl((x_n)_n,(x'_n)_n\bigr) = \lim_{\mathscr{U}}
  \frac{d(x_n,x'_n)}{s^{(n)}} \] defines a pseudometric on
$X^{\N}_{(b^{(n)})_n,(s^{(n)})_n}$.  The \defterm{asymptotic cone} of
$X$ with respect to $\mathscr{U}$ and the \defterm{basepoint sequence}
$(b^{(n)})_n$ and the \defterm{scaling sequence} $(s^{(n)})_n$ is the
metric space
$\AsCone_{\mathscr{U}}\bigl(X,(b^{(n)})_n,(s^{(n)})_n\bigr)$ obtained
from $(X^{\N}_{(b^{(n)})_n,(s^{(n)})_n}, d)$ by identifying $(x_n)_n$
and $(x'_n)_n$ whenever $d\bigl((x_n)_n,(x'_n)_n\bigr) = 0$.

We will need the following three propositions, which are consequences
of theorems 10.46 and 10.83 in Drutu and Kapovich \cite{Drutu:2018}.

\begin{prop}
  \proplabel{scaleinv} Let $X$ be a scale-invariant proper geodesic
  metric space with transitive isometry group.  Then every asymptotic
  cone of $X$ is isometric to $X$.
\end{prop}

\begin{prop}
  \proplabel{roughisom} Let $f\colon X \to Y$ be rough isometry of
  rough geodesic metric spaces.  If $\hat X$ is an asymptotic cone of
  $X$ then $f$ induces an isometry $\hat f \colon \hat X \to \hat Y$
  to some asymptotic cone $\hat Y$ of $Y$.
\end{prop}

\begin{prop}
  \proplabel{fixbp} Let $X$ be a rough geodesic metric space with
  cobounded isometry group.  Then any asymptotic cone of $X$ is
  isometric to an asymptotic cone of $X$ with constant basepoint
  sequence.
\end{prop}

\begin{thm}
  \thmlabel{chc} Let $Y$ be a $C$-coarsely injective space and let $X$
  be an asymptotic cone of $Y$.  Then $X$ is countably hyperconvex.
\end{thm}
\begin{proof}
  Let $\mathscr{U} \subset \mathcal{P}(\N)$ be an ultrafilter on $\N$,
  let $(y^{(n)})_{n \in \N}$ be a basepoint sequence in $Y$, let
  $(d^{(n)})_{n \in \N}$ be a scaling sequence $d^{(n)} \to \infty$.
  Let $X$ be the asymptotic cone of $Y$ for the ultrafilter
  $\mathscr{U}$, the basepoint sequence $(y^{(n)})_n$ and the scaling
  sequence $(d^{(n)})_n$.
  
  Let $(x_i,r_i)_{i \in \N}$ in $X \times \R_{\ge 0}$ satisfy
  $d(x_i,x_{i'}) \le r_i + r_{i'}$.  We will prove that the
  intersection of the balls $\bigcap_{i \in \N} B_{r_i}(x_i)$ is
  nonempty.  For $i \in \N$ let $x_i$ be represented by
  $(x_i^{(n)})_n \in \prod_{n \in \N} Y$.  For $j,i,i' \in \N$,
  let
  \[ S_{j,i,i'} = \Bigl\{n \in \N \sth
    \frac{1}{d^{(n)}}d(x_i^{(n)},x_{i'}^{(n)}) \le r_i + r_{i'} +
    \frac{1}{j} \Bigr\} \] and note that $S_{j,i,i'} \in \mathscr{U}$.
  Then the finite intersection
  \[ S_j = \bigcap_{i,i' \le j} S_{j,i,i'} = \Bigl\{n \in \N \sth
    \text{$\frac{1}{d^{(n)}}d(x_i^{(n)},x_{i'}^{(n)}) \le r_i + r_{i'}
      + \frac{1}{j}$ for $i,i' \le j$} \Bigr\} \] is also an element
  of $\mathscr{U}$ so the $S_j$ form a descending chain
  \[ S_1 \supseteq S_2 \supseteq S_3 \supseteq \cdots \] in $\mathscr{U}$.

  We inductively define a sequence
  $(x^{(n)})_{n \in \N} \in \prod_{n \in \N} Y$.  For
  $n \in \N \setminus S_1$, set $x^{(n)} = y^{(n)}$.  Assume we have
  defined $x^{(n)}$ for $n \in \N \setminus S_j$.  We now extend this
  definition to $\N \setminus S_{j+1}$.  Take
  $n \in S_j \setminus S_{j+1}$.  Then
  \[ d(x_i^{(n)},x_{i'}^{(n)}) \le \Bigl(r_i +
    \frac{1}{2j}\Bigr)d^{(n)} + \Bigl(r_{i'} +
    \frac{1}{2j}\Bigr)d^{(n)} \] for all $i,i' \le j$.  Since $Y$ is
  $C$-coarsely injective, we may choose $x^{(n)} \in Y$ such that
  \[ d(x^{(n)},x_i^{(n)}) \le \Bigl(r_i +
    \frac{1}{2j}\Bigr)d^{(n)} + C \] and so
  \begin{equation}
    \eqnlabel{have}
    \frac{1}{d^{(n)}}d(x^{(n)},x_i^{(n)}) \le r_i + \frac{1}{2j} +
    \frac{C}{d^{(n)}}
  \end{equation} for all $i \le j$.

  It remains to define $x^{(n)}$ for $n \in \bigcap_j S_j$.  Since
  \[ \bigcap_j S_j = \Bigl\{n \in \N \sth
    \text{$\frac{1}{d^{(n)}}d(x_i^{(n)},x_{i'}^{(n)}) \le r_i +
      r_{i'}$ for all $i,i'$} \Bigr\} \] if $n \in \bigcap_j S_j$ then
  \[ d(x_i^{(n)},x_{i'}^{(n)}) \le r_id^{(n)} + r_{i'}d^{(n)} \] for
  all $i,i'$.  Since $Y$ is $C$-coarsely injective, we may choose
  $x^{(n)} \in Y$ such that
  \[ d(x^{(n)},x_i^{(n)}) \le r_id^{(n)} + C \] and so
  \begin{equation}
    \eqnlabel{havetwo}
    \frac{1}{d^{(n)}}d(x^{(n)},x_i^{(n)}) \le r_i + \frac{C}{d^{(n)}}
  \end{equation} for all $i$.

  Having defined $(x^{(n)})_n$ we now show that it represents an
  element $x \in X$ for which $x \in \bigcap_{i \in \N} B_{r_i}(x_i)$,
  thus completing the proof of the countable hyperconvexity of $X$.
  Take $i,j \in \N$.  To prove that $x \in B_{r_i}(x_i)$ it will
  suffice to find a set $S \in \mathscr{U}$ such that
  \begin{equation}
    \eqnlabel{need}
    \frac{1}{d^{(n)}}d(x^{(n)},x_i^{(n)}) \le r_i + \frac{1}{j}
  \end{equation}
  for all $n \in S$.  Since $d^{(n)} \to \infty$, for each $j \in \N$
  there is $N \in \N$ such that $d^{(n)} \ge 2jC$ for all $n \ge N$.
  We claim that \peqnref{need} holds for all
  $n \in S = S_j \setminus \{1, 2, \ldots, N\}$.  Indeed, any
  $n \in S$ is either an element of $S_j' \setminus S_{j'+1}$ for some
  $j' \ge j$ or is an element of $\bigcap_j S_j$.  In the case
  $n \in S_{j'} \setminus S_{j'+1}$, we have
  \[ \frac{1}{d^{(n)}}d(x^{(n)},x_i^{(n)}) \le r_i + \frac{1}{2j'} +
    \frac{C}{d^{(n)}} \le r_i + \frac{1}{2j} + \frac{1}{2j} = r_i +
    \frac{1}{j} \] by \peqnref{have}.  In the case
  $n \in \bigcap_j S_j$ we have
  \[ \frac{1}{d^{(n)}}d(x^{(n)},x_i^{(n)}) \le r_i + \frac{C}{d^{(n)}}
    \le r_i + \frac{1}{2j} \le r_i + \frac{1}{j} \] by
  \peqnref{havetwo}.  Thus in either case we have \peqnref{need}, as
  required.
\end{proof}

\begin{cor}
  \corlabel{sephc} Any separable asymptotic cone of a coarsely
  injective space is injective.
\end{cor}
\begin{proof}
  Let $X$ be a separable asymptotic cone of a coarsely injective
  space.  A metric space is separable if and only if it is Lindel{\"
    o}f.  A space is \defterm{Lindel{\" o}f} if whenever a family
  $\mathcal{F}$ of closed sets has the property that every countable
  subfamily has nonempty intersection then $\mathcal{F}$ itself has
  nonempty intersection.

  By \Thmref{chc}, the asymptotic cone $X$ is countably hyperconvex.
  In particular, the asymptotic cone $X$ is complete and geodesic.
  Thus it suffices to show that any pairwise intersecting family of
  closed balls has nonempty total intersection.  If $\mathcal{F}$ is a
  family of pairwise intersecting closed balls in $X$ then every
  countable subfamily of $\mathcal{F}$ has nonempty intersection.
  Then, by the Lindel{\" o}f property, the total intersection of
  $\mathcal{F}$ is nonempty and we see that $X$ is injective.
\end{proof}

\section{Virtually nilpotent groups and injectivity}
\seclabel{nilpotent_injective}

In this section we use \Corref{sephc} along with results of Isbell on
injective metric spaces and results of Pansu on asymptotic cones of
nilpotent groups to prove that any virtually nilpotent coarsely
injective group is virtually abelian.  A group is \defterm{coarsely
  injective} if it acts metrically properly and coboundedly on a
coarsely injective space.  Any Helly group is coarsely injective.

Our main result in this section, \Thmref{nilpotent}, is also a
consequence of the semihyperbolicity of coarsely injective groups.
Lang showed that the injective hull of a coarsely injective space
has an isometry-invariant, bounded geodesic bicombing
\cite[Proposition~3.8]{Lang:2013} and, by Alonso and Bridson, this
implies that polycyclic subgroups of coarsely injective groups are
virtually abelian \cite[Theorem~4.1,
Theorem~7.1]{Alonso_Bridson:1995:semihyperbolic}.

Let $\mathfrak{g}$ be a
finite-dimensional Lie algebra.  The \defterm{lower central series}
\[ \mathfrak{g} = \mathfrak{g}_1 \supseteq \mathfrak{g}_2 \supseteq
  \mathfrak{g}_2 \supseteq \mathfrak{g}_3 \supseteq \cdots \] of
$\mathfrak{g}$ is defined inductively by
$\mathfrak{g}_i = [\mathfrak{g},\mathfrak{g}_{i-1}]$ for $i \ge 2$.
Recall that $\mathfrak{g}$ is \defterm{nilpotent} if and only if its
lower central series is eventually zero.  The \defterm{graded Lie
  algebra associated to} a finite-dimensional nilpotent Lie algebra
$\mathfrak{g}$ is the direct sum
$\mathfrak{g}_{\infty} = \bigoplus_{i=1}^{\infty} \mathfrak{g}_i /
\mathfrak{g}_{i+1}$ with the bracket defined by
$[x + \mathfrak{g}_{i+1}, y + \mathfrak{g}_{j+1}] = [x,y] +
\mathfrak{g}_{i+j+1} \in \mathfrak{g}_{i+j}/\mathfrak{g}_{i+j+1}$, for
$x \in \mathfrak{g}_i/\mathfrak{g}_{i+1}$ and
$y \in \mathfrak{g}_j/\mathfrak{g}_{j+1}$.  See Pansu
\cite{Pansu:1983}.

\begin{lem}
  \lemlabel{trivial_bracket} Let $\mathfrak{g}$ be a finite
  dimensional nilpotent Lie algerbra and let $\mathfrak{g}'$ be the
  graded Lie algebra associated to $\mathfrak{g}$.  If the Lie bracket
  on $\mathfrak{g}'$ is trivial then the Lie bracket on $\mathfrak{g}$
  is trivial.
\end{lem}
\begin{proof}
  If the Lie bracket on $\mathfrak{g}'$ is trivial then, for any
  $x,y \in \mathfrak{g}_1$, the bracket
  $[x + \mathfrak{g}_2, y + \mathfrak{g_2}] = [x,y] + \mathfrak{g}_3$
  is trivial in $\mathfrak{g}_2/\mathfrak{g}_3$, i.e., we have
  $[x,y] \in \mathfrak{g}_3$.  But then
  $\mathfrak{g}_2 = [\mathfrak{g}_1, \mathfrak{g}_1] \subseteq
  \mathfrak{g}_3$ so $\mathfrak{g}_2 = \mathfrak{g}_3$.  This implies
  that the lower central series of $\mathfrak{g}$ stabilizes at
  $\mathfrak{g}_2$ so, by nilpotence, we have
  $0 = \mathfrak{g}_2 = [\mathfrak{g}, \mathfrak{g}]$.
\end{proof}

To prove the next lemma we will need the following result of Isbell.

\begin{thm}[{Isbell \cite[Theorem~3.6]{Isbell:1964}}]
  \thmlabel{isbell} Let $X$ be an $n$-dimensional locally compact
  injective metric space.  Then there exists a subspace
  $B \subset X$ such that $B$ is homeomorphic to a ball of dimension
  $n$ and $B$ isometrically embeds in $(\R^n,|\cdot|_{\infty}\bigr)$.
\end{thm}

\begin{lem}
  \lemlabel{nilpinf} Let $G$ be a virtually nilpotent group acting
  metrically properly and coboundedly on a coarsely injective space
  $Y$.  Then any asymptotic cone of $Y$ is isometric to
  $\bigl(\R^n,|\cdot|_{\infty}\bigr)$.
\end{lem}
\begin{proof}
  Since $G$ acts metrically properly and coboundedly on a roughly
  geodesic metric space, it is finitely generated.  By
  \Propref{pbmetric}
  \[ d_G(g_1,g_2) = d_{Y}(g_1\cdot y_0,g_2\cdot y_0) +
    \delta(g_1, g_2) \] defines a roughly geodesic, left-invariant
  metric on $G$ and $g \mapsto g\cdot y_0$ is a rough isometry from
  $(G,d)$ to $Y$.  Since $G$ is virtually nilpotent, by Pansu
  \cite{Pansu:1983}, there is a unique asymptotic cone of $(G,d)$,
  independent of the choice of ultrafilter, basepoint sequence and
  scaling sequence.  Moreover, this asymptotic cone
  $(G_{\infty}, d_{\infty})$ is a simply connected nilpotent Lie group
  $G_{\infty}$, the metric $d_{\infty}$ is left-invariant and induces
  the manifold topology on $G_{\infty}$ and there exists a
  one-parameter family of Lie group automorphisms $\delta_t$ that are
  dilations of $d_{\infty}$, that is
  \[d_{\infty}\bigl(\delta_t(x),\delta_t(y)\bigr) = td_{\infty}(x,y)\]
  for any $t > 0$ and any $x,y \in G_{\infty}$.  By
  \Propref{roughisom}, it suffices to show that
  $(G_{\infty}, d_{\infty})$ is isometric to
  $\bigl(\R^n,|\cdot|_{\infty}\bigr)$.

  Since $G_{\infty}$ is simply connected and nilpotent, it is
  homeomorphic to $\R^n$, for some $n$
  \cite[Theorem~1.104]{knapp:lie_groups}.  Hence, by \Corref{sephc},
  the asymptotic cone $(G_{\infty},d_{\infty})$ is injective.  Then,
  by \Thmref{isbell} of Isbell and Brouwer's Invariance of Domain
  Theorem, the asymptotic cone $(G_{\infty},d_{\infty})$ contains an
  open set $U$ that is isometric to an open subspace of
  $\bigl(\R^n,|\cdot|_{\infty}\bigr)$.  Since $d_{\infty}$ induces the
  manifold topology, for any $x \in U$ there is some $\epsilon > 0$
  such that the ball $B_{\epsilon}(x)$ of radius $\epsilon$ centered
  at $x$ is contained in $U$ and so is isometric to a ball of radius
  $\epsilon$ in $\bigl(\R^n,|\cdot|_{\infty}\bigr)$.  Applying
  dilations and translations in $G_{\infty}$ and
  $\bigl(\R^n,|\cdot|_{\infty}\bigr)$ we see that, for any ball
  $B_r(x)$ of $G_{\infty}$ of any radius $r$ and any center $x$, there
  is an isometry from $B_r(x)$ to the ball $B_r(0)$ of radius $r$ of
  $\bigl(\R^n,|\cdot|_{\infty}\bigr)$ centered at the origin
  $0 \in \R^n$.  These balls are cubes with the $l_{\infty}$ metric.

  For any integer $k \ge 1$ let $B_k(1)$ be the ball of radius $k$
  centered at the identity $1 \in G_{\infty}$ and let $B_k(0)$ be the
  ball of radius $k$ centered at the origin $0$ in
  $\bigl(\R^n,|\cdot|_{\infty}\bigr)$.  Let
  $f_1 \colon B_1(1) \to B_1(0)$ be any isometry.  By Mankiewicz's
  generalization of the Mazur-Ulam Theorem \cite{Mankiewicz:1972}, for
  each $k > 1$ there is a unique isometry
  $f_k \colon B_k(1) \to B_k(0)$ extending $f_1$.  Then
  $f = \cup_{k \ge 1} f_k$ is an isometry from
  $\cup_{k \ge 1} B_k(1) = G_{\infty}$ to
  $\bigl(\R^n,|\cdot|_{\infty}\bigr)$.
\end{proof}

\begin{thm}
  \thmlabel{nilpotent} Let $G$ be a virtually nilpotent coarsely
  injective group.  Then $G$ is virtually abelian.
\end{thm}
\begin{proof}
  Since $G$ acts metrically properly and coboundedly on a roughly
  geodesic metric space, it is finitely generated.  Let $N < G$ be a
  finite index nilpotent subgroup of $G$.  Then $N$ is finitely
  generated.  The finite order elements of $N$ form a finite normal
  subgroup $\tor(N)$ \cite[Theorem~2.25 and
  Theorem~2.26]{clement:nilpotent}.  Let $N' = N/\tor(N)$ be the
  quotient group.  Let $M$ be the real Mal'cev completion of $N'$
  \cite{malcev:completion}.  Then $M$ is a simply connected nilpotent
  Lie group and $N'$ is a uniform lattice in $M$.  Let $\mathfrak{m}$
  be the Lie algebra of $M$.
  
  Let $Y$ be a coarsely injective space with a proper cocompact
  action of $G$.  By \Propref{pbmetric}
  \[ d_G(g_1,g_2) = d_{Y}(g_1\cdot y_0,g_2\cdot y_0) +
    \delta(g_1, g_2) \] defines a roughly geodesic, left-invariant
  metric on $G$ and $g \mapsto g\cdot y_0$ is a rough isometry from
  $(G,d)$ to $Y$.  By Pansu \cite{Pansu:1983}, there is a unique
  asymptotic cone of $(G,d)$, independent of the choice of
  ultrafilter, basepoint sequence and scaling sequence.  Moreover,
  this asymptotic cone $(G_{\infty}, d_{\infty})$ is a simply
  connected nilpotent Lie group $G_{\infty}$ and the metric
  $d_{\infty}$ is left-invariant.  Moreover, by Pansu
  \cite{Pansu:1983}, the Lie algebra $\mathfrak{g}_{\infty}$ of
  $G_{\infty}$ is the graded Lie algebra associated to $\mathfrak{m}$.
  
  Applying \Lemref{nilpinf} and \Propref{roughisom} we have that
  $(G_{\infty},d_{\infty})$ is isometric to
  $\bigl(\R^n,|\cdot|_{\infty}\bigr)$.  Then $G_{\infty}$ acts freely,
  transitively and continuously by isometries on
  $\bigl(\R^n,|\cdot|_{\infty}\bigr)$ and so we may view $G_{\infty}$
  as a connected subgroup of $\Isom\bigl(\R^n,|\cdot|_{\infty}\bigr)$.
  By the Mazur-Ulam Theorem the elements of
  $\Isom\bigl(\R^n,|\cdot|_{\infty}\bigr)$ are affine maps and so we
  have a map
  $\pi\colon \Isom\bigl(\R^n,|\cdot|_{\infty}\bigr) \to \GL_n(\R)$
  sending each affine isometry to its linear part.  The image of $\pi$
  is the group $\IsomAut\bigl(\R^n,|\cdot|_{\infty}\bigr)$ of linear
  isometries of $\bigl(\R^n,|\cdot|_{\infty}\bigr)$, which is finite
  and so totally disconnected.  Then since $G_{\infty}$ is connected,
  it is contained in the kernel of $\pi$.  That is, the action of
  $G_{\infty}$ on $\bigl(\R^n,|\cdot|_{\infty}\bigr)$ is by
  translations.  This action is faithful and so $G_{\infty}$ is
  abelian.  Hence the Lie bracket on $\mathfrak{g}_{\infty}$ is
  trivial and so, by \Lemref{trivial_bracket}, the Lie bracket on
  $\mathfrak{m}$ is trivial.  Then $M$ is abelian so that
  $N' = N/\tor(N)$ is abelian.  Hence $N$ is virtually abelian and so
  $G$ is virtually abelian.
\end{proof}

\section{Virtually abelian Helly groups}
\seclabel{abelian_helly}

In this section we fully characterize the virtually abelian groups
that are Helly.  In our main theorem stated below, the equivalence
between conditions \pitmref{ptgroup}, \pitmref{affisom},
\pitmref{tiling} and \pitmref{cocub} was proven by Hagen
\cite{Hagen:2014}.  The implication
$\pitmref{cocub} \Rightarrow \pitmref{ghelly}$ follows from results of
Bandelt \cite[Proposition~2.6]{Bandelt:1991} and of Polat and Pouzet
\cite[Proposition~3.1.2]{Polat:2001} via the connection between median
graphs and $\CAT(0)$ cube complexes \cite{Roller:1998, Gerasimov:1998,
  Chepoi:2000}.  The implication
$\pitmref{ginjective} \Rightarrow \pitmref{ptgroup}$ is our original
contribution.

\begin{thm}
  \thmlabel{main} Let $G$ be a finitely generated virtually abelian group.
  Let $P' < \GL_n(\R)$ be the point group of $G$.  The following
  conditions are equivalent.
  \begin{enumerate}
  \item \itmlabel{ghelly} $G$ is Helly.
  \item \itmlabel{ginjective} $G$ is coarsely injective.
  \item \itmlabel{ptgroup} There is a $\phi \in \GL_n(\R)$ such that
    $\phi P' \phi^{-1} < \IsomAut\bigl(\R^n,|\cdot|_{\infty}\bigr)$
    where $\IsomAut\bigl(\R^n,|\cdot|_{\infty}\bigr) < \GL_n(\R)$ is
    the group of isometric automorphisms of $\R^n$ with the supremum
    norm $|\cdot|_{\infty}$.
  \item \itmlabel{affisom} $G$ acts properly and cocompactly by affine
    isometries on $\bigl(\R^n,|\cdot|_{\infty}\bigr)$.
  \item \itmlabel{tiling} $G$ acts properly and cocompactly by cellular
    automorphisms on the standard cubical tiling of $\R^n$.
  \item \itmlabel{cocub} $G$ is cocompactly cubulated.
  \end{enumerate}
\end{thm}
\begin{proof}
  By the discussion in \Secref{crystals}, we have a morphism
  \[ \begin{tikzcd}
      1 \ar[r] & \Z^n \ar[r] \ar[d,hook] & G \ar[r] \ar[d] & P \ar[r] \ar[d, "1_P"] & 1 \\
      1 \ar[r] & \R^n \ar[r] & \R^n \rtimes P \ar[r] & P \ar[r] & 1
    \end{tikzcd} \] of short exact sequences where $P$ is finite, the
  map $\Z^n \hookrightarrow \R^n$ is the inclusion and the bottom row
  is the usual short exact sequence of the semidirect product.  Then
  the $P$-action $\alpha \colon P \to \GL_n(\R)$ on $\R^n$ is the
  composition of the $P$ action $P \to \GL_n(\Z)$ on $\Z^n$ with the
  inclusion $\GL_n(\Z) \hookrightarrow \GL_n(\R)$ and the point group
  $P'$ is the image of $\alpha$.

  We now prove the implication
  $\pitmref{ginjective} \Rightarrow \pitmref{ptgroup}$.  Let $Y$ be
  a coarsely injective space on which $G$ acts properly and
  cocompactly.  Choose a point $y_0 \in Y$.  By \Propref{pbmetric}
  \[ d_G(g_1,g_2) = d_{Y}(g_1\cdot y_0,g_2\cdot y_0) +
    \delta(g_1, g_2) \] defines a roughly geodesic, left-invariant
  metric on $G$ and $g \mapsto g\cdot y_0$ is a rough isometry from
  $(G,d)$ to $Y$.  For any $h \in G$, we have
  \begin{align*}
    d(hg_1h^{-1},hg_2h^{-1}) &= d(g_1h^{-1},g_2h^{-1}) \\
    &\le d(g_1h^{-1}, g_1) + d(g_1, g_2) + d(g_2,g_2h^{-1}) \\
    &= d(g_1,g_2) + 2d(1,h^{-1})
  \end{align*}
  and so the conjugation action of $G$ on itself is an action by rough
  isometries.  Hence $P$ acts on on $(\Z^n,d)$ by rough isometries.

  Let $|\cdot|_d$ be the norm on $\R^n$ obtained from $(\Z^n,d)$ as
  given by \Propref{mtonorm}.  Then the inclusion
  $(\Z^n,d) \hookrightarrow \bigl(\R^n,|\cdot|_d\bigr)$ is a rough
  isometry and the $P$-action on $\bigl(\R^n,|\cdot|_d\bigr)$ is by
  linear isometries.  Thus we have the following diagram
  \[ \bigl(\R^n,|\cdot|_d\bigr) \hookleftarrow (\Z^n,d)
  \hookrightarrow (G,d) \rightarrow Y \]
  in which all arrows are rough isometries.  By \Propref{roughisom}
  and \Propref{scaleinv} we see that any asymptotic cone of $Y$
  is isometric to $\bigl(\R^n,|\cdot|_d\bigr)$.  But, by
  \Lemref{nilpinf}, any asymptotic cone of $Y$ is isometric to
  $\bigl(\R^n,|\cdot|_{\infty}\bigr)$ so, there is an isometry
  $\phi \colon \bigl(\R^n,|\cdot|_d\bigr) \to
  \bigl(\R^n,|\cdot|_{\infty}\bigr)$.
  By the Mazur-Ulam Theorem, the isometry $\phi$ is an affine map.
  So, by composing with a translation, we may assume that $\phi$ is a
  linear isometry.  Then $\phi \in \GL_n(\R)$ and
  $\phi P' \phi^{-1} < \IsomAut \bigl(\R^n,|\cdot|_{\infty}\bigr)$.
  This completes the proof of
  $\pitmref{ginjective} \Rightarrow \pitmref{ptgroup}$.

  Next we prove $\pitmref{ptgroup} \Rightarrow \pitmref{affisom}$.
  Let $\phi \in \GL_n(\R)$ conjugate $P'$ into
  $\IsomAut\bigl(\R^n,|\cdot|_{\infty}\bigr)$.  Then the composition
  of $G \to \R^n \rtimes P$ with the map
  \begin{align*}
    \R_n \rtimes P &\to \R_n \rtimes \IsomAut\bigl(\R^n,|\cdot|_{\infty}\bigr) \\
    (r,p) &\mapsto \bigl(\phi(r),\phi \alpha(p) \phi^{-1}\bigr)
  \end{align*}
  is a proper and cocompact action of $G$ on the group
  $\R_n \rtimes \IsomAut\bigl(\R^n,|\cdot|_{\infty}\bigr)$ of affine
  isometries of $\bigl(\R^n,|\cdot|_{\infty}\bigr)$.

  To prove $\pitmref{affisom} \Rightarrow \pitmref{tiling}$, suppose
  $G$ acts properly and cocompactly by affine isometries on
  $\bigl(\R^n,|\cdot|_{\infty}\bigr)$.  Let
  $H_i = \bigl\{(x_1, x_2, \ldots, x_n) \in \R^n \sth x_i = 0\bigr\}$
  be the $i$th coordinate hyperplane of $\R^n$.  Consider the
  polyhedral cellulation of $\R^n$ induced by the set of all
  translates of the $H_i$ under the action of $G$.  Restricting the
  $G$-action to the image of $\Z^n \to G$ we see that the cells of the
  cellulation are bounded.  Since $G$ acts by affine isometries, each
  $g H_i$ is perpendicular to some $H_{i'}$ so the cells are all
  axis-aligned boxes.  Finally, since the action of $G$ is proper and
  cocompact, the $G$-orbit of the $0 \in \R^n$ is discrete and so the
  cellulation induced by the coordinate hyperplanes is isomorphic to
  the standard tiling of $\R^n$ by cubes.

  The $\pitmref{tiling} \Rightarrow \pitmref{cocub}$ implication is
  immediate.

  The implication $\pitmref{cocub} \Rightarrow \pitmref{ghelly}$
  follows from the theorem that if $X$ is a $\CAT(0)$ cube complex
  then the graph obtained from $X^0$ by joining any two vertices that
  are contained in a common cube of $X$ is a Helly graph
  \cite{Bandelt:1991, Polat:2001, Roller:1998, Gerasimov:1998,
    Chepoi:2000}.

  The $\pitmref{ghelly} \Rightarrow \pitmref{ginjective}$ implication is
  immediate.
\end{proof}

\begin{cor}
  The $3$-$3$-$3$-Coxeter group $G$ is not coarsely injective.  In
  particular, $G$ is not Helly.
\end{cor}
\begin{proof}
  Let $P' < \GL_2(\R)$ be the the point group of $G$.  Then $G$ has an
  element of order $3$ whereas
  $\IsomAut\bigl(\R^n,|\cdot|_{\infty}\bigr)$ is isomorphic to the
  dihedral group $D_4$ and so has no elements of order $3$.  Hence
  $P'$ cannot be conjugated into
  $\IsomAut\bigl(\R^n,|\cdot|_{\infty}\bigr)$.
\end{proof}

\begin{cor}
  There exists a group that is both systolic and $\CAT(0)$ but not
  coarsely injective.
\end{cor}

\appendix
\section{Short exact sequences and equivariant homomorphisms}

In this appendix, we prove a proposition about morphisms of short
exact sequences of groups.

Let \[ \begin{tikzcd}
    1 \ar[r] & N \ar[r] \ar[d, "\phi"] & G \ar[r] \ar[d, "f"] & Q \ar[r] \ar[d, "1_Q"] & 1 \\
    1 \ar[r] & M \ar[r] & H \ar[r] & Q \ar[r] & 1
  \end{tikzcd} \] be a commutative diagram of groups where the two
rows are short exact sequences.  Such a diagram is called a
\defterm{morphism of short exact sequences of groups}.  Then $G$ acts
on $N$ by conjugation as in $g \cdot n = gng^{-1}$ and $G$ acts on $M$
by conjugation via $G \to H$ as in $g \cdot m = f(g)mf(g)^{-1}$.
Moreover, viewing $N$ and $M$ as subgroups of $G$ and $H$, we have
\[ \phi(g \cdot n) = \phi(gng^{-1}) = f(gng^{-1}) = f(g)f(n)f(g)^{-1}
  = f(g)\phi(n)f(g)^{-1} = g \cdot \phi(n) \] and
\[ n \cdot m = f(n)mf(n)^{-1} = \phi(n)m\phi(n)^{-1} \] so $\phi$ is
$G$-equivariant and the $G$ action on $M$ extends the $N$ action on
$M$ by conjugation via $\phi$.  The following proposition gives a
converse.

\begin{prop}
  \proplabel{sesmap} Let \[ 1 \to N \to G \to Q \to 1 \] be a short
  exact sequence of groups.  Consider the left $G$-action by
  conjugation on $N$ given by $g \cdot n = gng^{-1}$.  Suppose $G$
  acts on some other group $M$ by automorphisms and let
  $\phi \colon N \to M$ be a $G$-equivariant homomorphism such that
  for all $n \in N$ and $m \in M$ we have
  $n \cdot m = \phi(n)m\phi(n)^{-1}$.  Then, up to isomorphism, there
  is a unique morphism of short exact sequences
  \[ \begin{tikzcd}
      1 \ar[r] & N \ar[r] \ar[d, "\phi"] & G \ar[r] \ar[d, "\hat\phi"] & Q \ar[r] \ar[d, "1_Q"] & 1 \\
      1 \ar[r] & M \ar[r] & G_{\phi} \ar[r] & Q \ar[r] & 1
    \end{tikzcd} \] such that
  \[ g \cdot m = \hat\phi(g)m\hat\phi(g)^{-1} \] for all $g \in G$.
\end{prop}
\begin{proof}
  Let
  $K = \Bigl\{\bigl(\phi(n)^{-1},n\bigr) \sth n \in N \Bigr\} \subset
  M \rtimes G$.  Then, since
  \[ \bigl(\phi(n_1)^{-1},n_1\bigr)^{-1} =
    \bigl(n_1^{-1} \cdot \phi(n_1),n_1^{-1}\bigr)
    = \bigl(\phi(n_1^{-1})\phi(n_1)\phi(n_1),n_1^{-1}\bigr)
    = \bigl(\phi(n_1),n_1^{-1}\bigr)\] and
  \begin{align*}
    \bigl(\phi(n_1)^{-1},n_1\bigr)\bigl(\phi(n_2)^{-1},n_2\bigr)
  &= \bigl(\phi(n_1)^{-1}(n_1 \cdot \phi(n_2)^{-1}),n_1n_2\bigr) \\
  &= \bigl(\phi(n_1)^{-1}\phi(n_1)\phi(n_2)^{-1}\phi(n_1)^{-1}),n_1n_2\bigr) \\
  &= \bigl(\phi(n_2)^{-1}\phi(n_1)^{-1}),n_1n_2\bigr) \\
  &= \bigl(\phi(n_1n_2)^{-1},n_1n_2\bigr)
  \end{align*}
  $K$ is a subgroup of $M \rtimes G$.  Moreover
  \begin{align*}
    (m,g)\bigl(\phi(n)^{-1},n\bigr)(m,g)^{-1}
    &= \Bigl(m\bigl(g\cdot \phi(n)^{-1}\bigr), gn\Bigr)
      (g^{-1} \cdot m^{-1}, g^{-1}) \\
    &= \bigl(m\phi(g\cdot n)^{-1}(gng^{-1} \cdot m^{-1}), gng^{-1}\bigr) \\
    &= \bigl(m\phi(g\cdot n)^{-1}\phi(gng^{-1})m^{-1}\phi(gng^{-1})^{-1}, gng^{-1}\bigr) \\
    &= \bigl(m\phi(g\cdot n)^{-1}\phi(g \cdot n)m^{-1}\phi(gn^{-1}g^{-1}), gng^{-1}\bigr) \\
    &= \bigl(\phi(gn^{-1}g^{-1}), gng^{-1}\bigr)
  \end{align*}
  and so $K$ is a normal subgroup of $M \rtimes G$.

  Since $M$ and $K$ are disjoint in $M \rtimes G$ the map
  $M \rtimes G \to (M \rtimes G) / K$ is an embedding and so
  \[ 1 \to M \to (M \rtimes G) / K \to (M \rtimes G) / MK \to 1 \] is
  a short exact sequence.  By the definition of $K$, the square
  \[ \begin{tikzcd}
      N \ar[r] \ar[d, "\phi"] & \ar[d] G \\
      M \ar[r] & (M \rtimes G) / K
    \end{tikzcd} \] commutes and so the diagram
  \[ \begin{tikzcd}
      1 \ar[r] & N \ar[r] \ar[d, "\phi"] & G \ar[r] \ar[d] & Q \ar[r] \ar[d]
      & 1 \\
      1 \ar[r] & M \ar[r] & (M \rtimes G) / K \ar[r] & (M \rtimes G) /
      MK \ar[r] & 1
    \end{tikzcd} \] commutes, where $Q \to (M \rtimes G) / MK$ is
  defined by $gN \mapsto (1,g)MK$.

  Taking $G_{\alpha} = (M \rtimes G) / K$ we need to show that
  $Q \to (M \rtimes G) / MK$ is an isomorphism.  Since
  $(m,g)MK = (1,g)MK$ the map $Q \to (M \rtimes G) / MK$ is
  surjective.  If $gN$ is in the kernel of $Q \to (M \rtimes G) / MK$
  then $(1,g) \in MK$ so
  $(1,g) = (m,1)\bigl(\phi(n)^{-1},n\bigr) =
  \bigl(m\phi(n)^{-1},n\bigr)$ for some $m \in M$ and $n \in N$.  Then
  $m = \phi(n)$ and $g = n$ so $g \in N$.  Hence
  $Q \to (M \rtimes G) / MK$ is also injective.

  It remains to prove uniqueness.  Suppose
  \begin{equation}
    \eqnlabel{morph}
    \begin{tikzcd}
      1 \ar[r] & N \ar[r] \ar[d, "\phi"] & G \ar[r] \ar[d, "\hat\phi'"] & Q \ar[r] \ar[d, "1_Q"] & 1 \\
      1 \ar[r] & M \ar[r] & G_{\phi}' \ar[r] & Q \ar[r] & 1
    \end{tikzcd}
  \end{equation}
  is some other morphism of short exact sequences such that
  \[ g \cdot m = \hat\phi'(g)m\hat\phi'(g)^{-1} \] for all $g \in G$.
  Consider the function $\psi\colon M\rtimes G \to G_{\phi}'$ given by
  $\psi(m,g) = m\hat\phi'(g)$.  By the computation
  \begin{align*}
    \psi(m_1,g_1)\psi(m_2,g_2)
    &= m_1\hat\phi'(g_1)m_2\hat\phi'(g_2) \\
    &= m_1\hat\phi'(g_1)m_2\hat\phi'(g_1)^{-1}\hat\phi'(g_1)\hat\phi'(g_2) \\
    &= m_1(g_1 \cdot m_2)\hat\phi'(g_1g_2) \\
    &= \psi\bigl(m_1(g_1 \cdot m_2),g_1g_2\bigr) \\
    &= \psi\bigl((m_1,g_1)(m_2,g_2)\bigr)
  \end{align*}
  $\psi$ is a homomorphism.  The composition
  $M \to M \rtimes G \to G_{\phi}'$ is the inclusion of $M$ and the
  composition $G \to M \rtimes G \to G_{\phi}' \to Q$ is the
  composition $G \to G_{\phi}' \to Q$ which, by exactness and
  commutativity in \peqnref{morph}, is surjective.  So the image of
  $\psi$ contain $M$ and intersects every coset of $M$ and so $\psi$
  is surjective.  The kernel of $\psi$ is
  $\ker\psi = \bigl\{(m,g) \sth \hat\phi'(g) = m^{-1} \bigr\}$ but if
  $\phi'(g) \in M$ then, by exactness and commutativity in
  \peqnref{morph}, we have $g \in N$.  So
  \[ \ker\psi = \Bigl\{\bigl(\phi(n)^{-1},n\bigr) \sth n \in N \Bigr\}
    = K \] so $\psi$ descends to an isomorphism
  $\bar\psi \colon (M \rtimes G) / K \to G_{\phi}'$ and we have an
  isomorphism of short exact sequences
  \[ \begin{tikzcd}
      1 \ar[r] & M \ar[r] \ar[d, "1_M"]
      & (M \rtimes G) / K \ar[r] \ar[d, "\bar\psi"]
      & (M \rtimes G) / MK \ar[r] \ar[d] & 1 \\
      1 \ar[r] & M \ar[r] & G_{\phi}' \ar[r] & Q \ar[r] & 1
    \end{tikzcd} \] through which the morphism \peqnref{morph}
  factors.
\end{proof}

\bibliographystyle{abbrv}
\bibliography{nima,syshel}

\begin{thebibliography}{10}

\bibitem{Alonso_Bridson:1995:semihyperbolic}
J.~M. Alonso and M.~R. Bridson.
\newblock Semihyperbolic groups.
\newblock {\em Proc. London Math. Soc. (3)}, 70(1):56--114, 1995.

\bibitem{Aronszajn:extension}
N.~Aronszajn and P.~Panitchpakdi.
\newblock Extension of uniformly continuous transformations and hyperconvex
  metric spaces.
\newblock {\em Pacific J. Math.}, 6:405--439, 1956.

\bibitem{Bandelt:1991}
H.-J. Bandelt and M.~van~de Vel.
\newblock Superextensions and the depth of median graphs.
\newblock {\em J. Combin. Theory Ser. A}, 57(2):187--202, 1991.

\bibitem{bieberbach:crystallographic1}
L.~Bieberbach.
\newblock {\" U}ber die {B}ewegungsgruppen der {E}uklidischen {R}{\" a}ume.
\newblock {\em Mathematische Annalen}, 70(3):297--336, 1911.

\bibitem{bieberbach:crystallographic2}
L.~Bieberbach.
\newblock {\" U}ber die {B}ewegungsgruppen der {E}uklidischen {R}{\" a}ume
  ({Z}weite {A}bhandlung.) {D}ie {G}ruppen mit einem endlichen
  {F}undamentalbereich.
\newblock {\em Mathematische Annalen}, 72(3):400--412, 1912.

\bibitem{brown:cohomology}
K.~S. Brown.
\newblock {\em Cohomology of groups}, volume~87 of {\em Graduate Texts in
  Mathematics}.
\newblock Springer-Verlag, New York-Berlin, 1982.

\bibitem{Burago:1992}
D.~Y. Burago.
\newblock Periodic metrics.
\newblock In {\em Representation theory and dynamical systems}, volume~9 of
  {\em Adv. Soviet Math.}, pages 205--210. Amer. Math. Soc., Providence, RI,
  1992.

\bibitem{Chalopin:2020}
J.~Chalopin, V.~Chepoi, A.~Genevois, H.~Hirai, and D.~Osajda.
\newblock {H}elly groups.
\newblock {\em Preprint, arXiv:1904.09060}, 2020.

\bibitem{Chepoi:2000}
V.~Chepoi.
\newblock Graphs of some {${\rm CAT}(0)$} complexes.
\newblock {\em Adv. in Appl. Math.}, 24(2):125--179, 2000.

\bibitem{chepoi:packing:2017}
V.~Chepoi, B.~Estellon, and G.~Naves.
\newblock Packing and covering with balls on {B}usemann surfaces.
\newblock {\em Discrete Comput. Geom.}, 57(4):985--1011, 2017.

\bibitem{clement:nilpotent}
A.~E. Clement, S.~Majewicz, and M.~Zyman.
\newblock {\em The theory of nilpotent groups}.
\newblock Birkh\"{a}user/Springer, Cham, 2017.

\bibitem{Drutu:2018}
C.~Dru\c{t}u and M.~Kapovich.
\newblock {\em Geometric group theory}, volume~63 of {\em American Mathematical
  Society Colloquium Publications}.
\newblock American Mathematical Society, Providence, RI, 2018.
\newblock With an appendix by Bogdan Nica.

\bibitem{espinola_khamsi:hyperconvex}
R.~Esp\'{\i}nola and M.~A. Khamsi.
\newblock Introduction to hyperconvex spaces.
\newblock In {\em Handbook of metric fixed point theory}, pages 391--435.
  Kluwer Acad. Publ., Dordrecht, 2001.

\bibitem{Fujiwara:2015}
K.~Fujiwara.
\newblock Asymptotically isometric metrics on relatively hyperbolic groups and
  marked length spectrum.
\newblock {\em J. Topol. Anal.}, 7(2):345--359, 2015.

\bibitem{Gerasimov:1998}
V.~Gerasimov.
\newblock Fixed-point-free actions on cubings.
\newblock {\em Siberian Adv. Math.}, 8(3):26--58, 1998.

\bibitem{Hagen:2014}
M.~F. Hagen.
\newblock Cocompactly cubulated crystallographic groups.
\newblock {\em J. Lond. Math. Soc. (2)}, 90(1):140--166, 2014.

\bibitem{Huang:2019}
J.~Huang and D.~Osajda.
\newblock Helly meets {G}arside and {A}rtin.
\newblock {\em Invent. Math.}, 225(2):395--426, 2021.

\bibitem{Isbell:1964}
J.~R. Isbell.
\newblock Six theorems about injective metric spaces.
\newblock {\em Comment. Math. Helv.}, 39:65--76, 1964.

\bibitem{knapp:lie_groups}
A.~W. Knapp.
\newblock {\em Lie groups beyond an introduction}, volume 140 of {\em Progress
  in Mathematics}.
\newblock Birkh\"{a}user Boston, Inc., Boston, MA, 1996.

\bibitem{Lang:2013}
U.~Lang.
\newblock Injective hulls of certain discrete metric spaces and groups.
\newblock {\em J. Topol. Anal.}, 5(3):297--331, 2013.

\bibitem{malcev:completion}
A.~I. Malcev.
\newblock On a class of homogeneous spaces.
\newblock {\em Amer. Math. Soc. Translation}, 1951(39):33, 1951.

\bibitem{Mankiewicz:1972}
P.~Mankiewicz.
\newblock On extension of isometries in normed linear spaces.
\newblock {\em Bull. Acad. Polon. Sci. S\'{e}r. Sci. Math. Astronom. Phys.},
  20:367--371, 1972.

\bibitem{Pansu:1983}
P.~Pansu.
\newblock Croissance des boules et des g\'{e}od\'{e}siques ferm\'{e}es dans les
  nilvari\'{e}t\'{e}s.
\newblock {\em Ergodic Theory Dynam. Systems}, 3(3):415--445, 1983.

\bibitem{petyt_spriano:unbounded:2023}
H.~Petyt and D.~Spriano.
\newblock Unbounded domains in hierarchically hyperbolic groups.
\newblock {\em Groups Geom. Dyn.}, 17(2):479--500, 2023.

\bibitem{Polat:2001}
N.~Polat.
\newblock Convexity and fixed-point properties in {H}elly graphs.
\newblock {\em Discrete Math.}, 229(1-3):197--211, 2001.
\newblock Combinatorics, graph theory, algorithms and applications.

\bibitem{Roller:1998}
M.~Roller.
\newblock {\em Poc sets, median algebras and group actions}.
\newblock Habilitationsschrift, University of Regensburg, 1998.

\bibitem{zassenhaus:crystallographic}
H.~Zassenhaus.
\newblock \"{U}ber einen {A}lgorithmus zur {B}estimmung der {R}aumgruppen.
\newblock {\em Comment. Math. Helv.}, 21:117--141, 1948.

\end{thebibliography}
\end{document}
